\documentclass[letterpaper, 10 pt, conference]{ieeeconf}  

\IEEEoverridecommandlockouts                              

\usepackage{amsmath}
\usepackage{amsfonts}
\usepackage{amssymb}
\usepackage{amsthm}
\usepackage{color}
\newcommand\hGA{\mathbf{\Gamma}}

\theoremstyle{plain}
\newtheorem{lemma}{Lemma}
\newtheorem{thm}{Theorem}
\newtheorem{corollary}{Corollary}
\newtheorem{proposition}{Proposition}

\theoremstyle{definition}

\newtheorem{example}{Example}

\theoremstyle{remark}
\newtheorem{remark}{Remark}

\newcommand{\expt}[1]{\ensuremath{\langle #1 \rangle}{}}

\newcommand{\gks}{{\sc gksl}}
\newcommand{\diag}{\operatorname{diag}}
\newcommand{\posdiag}[1]{ \mathrm D({#1})}
\newcommand{\pos}[1]{\mathcal D({#1})}

\newcommand{\fk}{\mathfrak{k}}

\newcommand{\bK}{\mathbf{K}}

\newcommand{\su}{\mathfrak{su}}

\title{\LARGE \bf
Reachable Sets from Toy Models to Controlled Markovian Quantum Systems*
}

\author{Gunther Dirr$^{1}$, Frederik vom Ende$^{2,3}$ and Thomas Schulte-Herbr{\"u}ggen$^{2,3}$
\thanks{*\;This work was supported in part 
by the Bavarian excellence network {\sc enb}
via the \mbox{International} PhD Programme of Excellence
{\em Exploring Quantum Matter} ({\sc exqm}).}
\thanks{$^{1}$G.D.~is with the Faculty of Mathematics, University of W\"urzburg, Emil-Fischer-Strasse 40, 97074 W\"urzburg, Germany
}%
\thanks{$^{2}$F.v.E.~and T.S-H.~are with the Department of Chemistry, Technische Universit{\"a}t M{\"u}nchen, Lichtenbergstrasse 4, 85747 Garching and with the}%
\thanks{$^{3}$Munich Centre for Quantum Science and Technology (MCQST),  Schellingstrasse 4, 80799~M{\"u}nchen, Germany
        {\tt\small frederik.vom-ende@tum.de}}%
\\{\small \em date: \today} 
}

\begin{document}


\maketitle

\pagestyle{headings}
\setcounter{page}{1}
\pagenumbering{arabic}


\begin{abstract}
In the framework of bilinear control systems, we present reachable sets of coherently controllable open 
quantum systems with switchable coupling to a thermal bath of arbitrary temperature $T \geq 0$. 
The core problem boils down to studying points in the standard simplex amenable to two types 
of controls that can be used interleaved:
\begin{itemize}
\item[(i)] permutations within the simplex,\vspace{2pt}
\item[(ii)] contractions by a dissipative one-parameter semigroup.\vspace{1pt}
\end{itemize}
Our work illustrates how the solutions of the core problem pertain to the reachable set of the original
controlled Markovian quantum system. We completely characterize the case $T=0$ and present inclusions
for $T>0$.
\end{abstract}

\section{INTRODUCTION}
Quantum systems theory and control engineering is a corner stone to unlock the potential of many
quantum devices in view of emerging technologies~\cite{DowMil03,Roadmap2015}.

To ensure well-posedness of a large class of control tasks, it is advisable
to check first whether the desired target state is within the reachable set of the
dynamic system. 
Here we show how reachability problems of (finite-dimensional) Markovian open
quantum systems can be reduced to studying hybrid control systems on the standard
simplex of $\mathbb R^n$. Our starting point is a bilinear control system \cite{Elliott09} of the
form
\begin{equation}\label{eq:bilin}
\dot \varphi(t) = -(A + \sum\nolimits_j u_j(t) B_j) \varphi(t)\,,\quad \varphi(0) =\varphi_0\,,
\end{equation}
where as usual $A$ denotes an uncontrolled drift, while the control terms consist of (piecewise
constant) control amplitudes $u_j(t)\in\mathbb R$ and control operators $B_j$. The
state $\varphi(t)$ may be thought of as (vectorized) density operator. The corresponding system Lie algebra, 
which provides the crucial tool for analysing controllability and accessibility questions, reads 
$\fk:=\expt{A, B_j\,|\, j=0,1,\dots,m}_{\sf Lie}$.

For ``closed'' quantum systems, i.e.~systems which do not interact with their environment,
the matrices $A$ and $B_j$ involved are skew-hermitian and thus it is known 
\cite{SJ72,JS72,Bro72,dAll08,DiHeGAMM08}
that the reachable set of \eqref{eq:bilin} is given by the orbit of the initial state under the
action of the dynamical systems group $\bK:=\expt{\exp \fk}$, provided $\bK$ is a compact subgroup 
of the unitary group.

More generally, for ``open'' systems undergoing Markovian dissipation,
the reachable set takes the form of a (Lie) semigroup orbit \cite{DHKS08}. -- Here we address an 
intermediate scenario with coherent controls $\{B_j\}_{j=1}^m$ and a bang-bang switchable 
dissipator $B_0$, the latter being motivated by recent experimental progress
\cite{Mart09,Mart13,Mart14,McDermott_TunDissip_2019} as described in~\cite{BSH16}.


\section{Specification of the Toy Model}
\label{sec:toy_model}

Under these assumptions and some further invariance condition
one can simplify the reachability analysis of \eqref{eq:bilin} to a core problem (dubbed \/`toy model\/' henceforth)
on the standard simplex
\begin{equation*}
\Delta^{n-1}:=\big\{ x\in\mathbb R_+^n\,|\, {\textstyle\sum}_{i=1}^nx_i=1\big\}\,.
\end{equation*}
In order to make the main features match the quantum dynamical context (described in Sec.~\ref{sec:control} below), 
let us fix the following stipulations for the toy model: Its controls shall amount to permutation matrices acting
instantaneously on the entries of $x(t)$ and a continuous-time one-parameter semigroup $(e^{-tB_0})_{t\in\mathbb R^+}$
of stochastic maps with a unique fixed point $d$ in $\Delta^{n-1}$. As $(e^{-tB_0})_{t\in\mathbb R^+}$ results from the 
restriction of the bang-bang switchable dissipator $B_0$, with abuse of notation we will denote its infinitesimal
generator again by $B_0$. The `\/{\em equilibrium state}\/' $d$ will be defined explicitly in Eq.~\eqref{eq:gibbs_vec}
by the system parameters and the absolute temperature $T\geq 0$ of an external bath.

Altogether, this yields what we call the `\/toy model\/' in the sequel. 
More precisely, these stipulations
suggest the following hybrid/impulsive toy model $\Lambda$ on $\Delta^{n-1} \subset \mathbb R^n$, 
cf.~\cite{book_impulsive89,Leela1991,book_HybridSytems96}: 
\begin{equation}\label{eq:control-simplex_evolution}
\begin{split}
&\dot{x}(t)  = -B_0 x(t)\,,\quad x(t_k) = x_k\,, \quad t \in [t_k,t_{k+1})\,,\\
& x_0  \in \Delta^{n-1}\,, \quad x_{k+1} = \pi_k e^{-(t_{k+1}-t_k)B_0}x_k\,, \quad k\geq 0\,,
\end{split}
\end{equation}
where the upper line describes the continuous-time evolution and the lower line the discrete-time part.
The switching sequence $0 =: t_0 \leq t_1 \leq t_2 \leq \dots$ and the permutation matrices $\pi_k$ are regarded
as controls for \eqref{eq:control-simplex_evolution}. For simplicity, we assume that the switching points do not 
accumulate on finite intervals. The reachable sets of $\Lambda$
\begin{equation*}
\mathfrak{reach}_\Lambda(x_0) := \{x(t) \,|\,
\text{$x(\cdot)$ is a solution of \eqref{eq:control-simplex_evolution}, $t \geq 0$}\}
\end{equation*}
allow for the following characterisation
\begin{equation*}
\mathfrak{reach}_\Lambda(x_0) =  {\mathcal S}_\Lambda x_0\,,
\end{equation*}
where ${\mathcal S}_\Lambda \subset \mathbb R^{n\times n}$ is the contraction semigroup generated 
by $(e^{-tB_0})_{t\in\mathbb R_+}$ and the set of all permutation matrices $\pi$.

\section{Main Results}

Henceforth, let $\hGA$ stand for a \gks-operator acting on complex $n \times n$ matrices, see
Eq.~\eqref{eq:lindblad_V}. Then $B_0$ in Eq.~\eqref{eq:bilin} can be regarded as its matrix representation
(obtained, e.g., via the Kronecker formalism \cite[Chap.~4]{HJ2}). If $\hGA$ leaves the set of diagonal matrices
invariant---a case we are primarily interested in---we denote by abuse of notation the corresponding matrix 
representation (obtained via $x \mapsto \operatorname{diag}(x)$)
again by $B_0$. If different $\hGA$ are involved we write $B_0(\hGA)$ to avoid confusion.
--- Within this picture, our main results can be sketched as follows.

\medskip
For $n\in\mathbb N$, consider the $n$-level toy model $\Lambda_0$ (cf.~Sec.~\ref{sec:toy_model}) 
with controls by permutations as above and an infinitesimal generator $B_0$ which results from a
dissipative coupling to a bath of temperature \mbox{$T=0$} 
(i.e.~$\hGA = \hGA_0$ is generated by a single $V = \sigma_+$, cf.~Eq.~\eqref{eq:sigma}).

\begin{thm}\label{thm_1}
Then the closure of the reachable set of any initial state $x_0 \in \Delta^{n-1}$ under the dynamics of
$\Lambda_0$ exhausts the full standard simplex, i.e.
$$
\overline{\mathfrak{reach}_{\Lambda_0}(x_0)}=\Delta^{n-1}\,.
$$
\end{thm}

Moving from a single $n$-level system (qu{\em d}it) 
to a tensor product of $m$ such $n$-level systems gives diagonal states
$x_0 \in \Delta^{n^m-1} \subset ({\mathbb R}^n)^{\otimes m}$. If the bath of temperature $T=0$ is coupled to
just one (say the last) of the $m$ qu{\em d}its, $\hGA_0$ is generated by $V :=I_{n^{m-1}}\otimes\sigma_+$ in Eq.~\eqref{eq:lindblad_V} and one obtains the following generalization.

\begin{thm}\label{thm_2}
The statement of Theorem~1 holds analogously for all initial states
$x_0\in\Delta^{n^m-1}$.
\end{thm}

In the sequel we refer to the standard concept and notation~($\prec$) of majorisation \cite{MarshallOlkin,Ando89}
and denote by
$d\in\Delta^{n-1}$ 
the unique attractive fixed point of the evolution generated by the dissipator $\hGA = \hGA_d$ for temperature
$T>0$ (see Sec.~\ref{sec:particular_models} with $\hGA_d$ comprising the generators $\sigma_-^d$ and $\sigma_+^d$ 
as in Prop.~\ref{thm_bath}).

\begin{thm}\label{thm_3}
Again allowing for permutations as controls interleaved with dissipation resulting from $B(\hGA_d)$  
 one obtains for the reachable set of the corresponding 
toy model $\Lambda_d$
\begin{equation*}
\mathfrak{reach}_{\Lambda_d}(d)\subseteq \lbrace x\in\Delta^{n-1}\,|\, x\prec d\rbrace\,.
\end{equation*}
\end{thm}

\noindent
The current results extend the qubit picture of~\cite{BSH16} to $n$-level systems, and 
even more generally to systems of $m$ qu{\em d}its.

For some mathematical statements contained within this manuscript we shall only sketch the ideas of 
how to prove them (denoted by \textit{``Sketch of Proof''}). 

\bigskip

\section{Relation of Controlled Quantum Systems to Toy Models}\label{sec:control}
Before proving the main theorems, we interpret our toy model in terms of open quantum systems.
Let $\pos{n}$ denote the set of all $n \times n$ density matrices (positive 
semi-definite matrices of trace 1) and $\mathcal L(\mathbb C^{n\times n})$ the set of all linear 
operators acting on complex $n\times n$-matrices. Then $\hGA\in\mathcal L(\mathbb C^{n\times n})$ 
with 
\begin{equation}\label{eq:lindblad_V}
\hGA(\rho):=\sum\nolimits_k\Big( \tfrac12 \big(V_k^\dagger V_k \rho+\rho V_k^\dagger V_k\big)-V_k\rho V_k^\dagger \Big)
\end{equation}
and arbitrary $V_k\in\mathbb C^{n\times n}$ will be called {\em Gorini-Kossakowski-Sudarshan-Lindblad} (\gks) {\em 
operator}  \cite{GKS76,Lind76}. It induces a linear ordinary differential equation ({\sc ode})
\begin{equation}\label{eq:diss_evolution}
\dot{\rho}(t)=-\hGA(\rho(t))\,,\quad \rho(0)=\rho_0\in\mathbb C^{n\times n}\,,
\end{equation}
whose solution reads $\rho(t)=e^{-t\hGA}\rho_0$ for all $t\in\mathbb R_+$. 
As $\hGA$ is of \gks-form, $(e^{-t\hGA})_{t\in\mathbb R_+}$ constitutes a one-parameter semigroup of completely 
positive, trace-preserving linear maps acting on $\mathbb C^{n\times n}$, cf.~\cite[Thm.~2.2]{GKS76}. This implies
that $(e^{-t\hGA})_{t\in\mathbb R_+}$ is a contraction semigroup which leaves $\pos{n}$ invariant. In particular,
$\|e^{-t\hGA}\| = 1$ for all $t\in\mathbb R_+$ when $\mathbb C^{n\times n}$ is equipped with the trace norm 
$\|\rho\|_1=\operatorname{tr}(\sqrt{\rho^\dagger \rho})$ \cite[Thm.~2.1]{wolf06}.\medskip

Next, let us extend \eqref{eq:diss_evolution} by coherent controls to a control system $\Sigma$ of the form
\begin{equation}\label{eq:control-diss_evolution}
\dot{\rho}(t)= -{\rm i}\Big[H_0 + \sum_{j=1}^m u_k(t) H_j,\rho(t)\Big] - \gamma(t) \hGA(\rho(t))\,,
\end{equation}
where all $H_j$ are (traceless) hermitian and $\gamma$ is a bang-bang switching function, i.e.~$\gamma(t) \in \{0,1\}$.
In general, an analytic description of the reachable sets of \eqref{eq:control-diss_evolution} is rather
challenging and in higher-dimensional cases almost impossible. However, there are a few scenarios which allow
partial results or even a complete characterization:

\begin{itemize}
\item 
If $\hGA(I_n) = 0$ (which is equivalent to $(e^{-t\hGA})_{t\in\mathbb R_+}$ being a
semigroup of {\em unital} quantum channels), then for any density matrix $\rho_0\in\pos{n}$ 
one has the estimate \cite{Ando89,Yuan10}
\begin{equation}\label{eq:reach-unital}
\mathfrak{reach}_\Sigma(\rho_0) \subseteq \{\rho \in \pos{n} \,|\, \rho \prec \rho_0 \}\,.
\end{equation}
\item
If $\hGA$ is of Kraus rank one, i.e.~$\hGA$ is generated by a single $V$,
and moreover if $V$ is normal, then one has (up to closure) equality in \eqref{eq:reach-unital} 
whenever the following assumption is satisfied:
\end{itemize}

\noindent
\textbf{Assumption UC+S:} The unitary part of \eqref{eq:control-diss_evolution} is {\em unitarily controllable},
i.e.~$\expt {{\rm i}H_j\,|\,j=0,\dots, m}_{\sf Lie} = \su(n)$ and the switching  function $\gamma(t)$ acts as
additional control. Here $\su(n)$ denotes the Lie algebra of all (traceless) skew-hermitian
$n \times n$ matrices.

\medskip
\noindent
 Alternatively, the assumption (FUC) below---which is unrealistic from the point of view of physics---leads to the same result.

\medskip
\noindent
\textbf{Assumption FUC:} The unitary part of \eqref{eq:control-diss_evolution} is {\em fully Hamiltonian controllable}
\cite{DHKS08}, i.e.~the Lie algebra generated by the control operators ${\rm i}H_j$ (without the drift ${\rm i}H_0$)
satisfies $\expt {{\rm i}H_j\,|\,j=1,\dots, m}_{\sf Lie} = \su(n)$, there are no restrictions on the controls
$u_j(t) \in \mathbb R$, and $\gamma(t) = \gamma_0 > 0$ for all $t \geq 0$.

 \medskip
\noindent
Recall that ``controllability of the unitary part'' is meant in the sense that the (lifted) 
bilinear system
\begin{equation*} 
\dot{X}(t) = - {\rm i}\Big(H_0 + \sum_{j=1}^m u_j(t) H_j \Big)X(t)\,,\quad X(0) = {\rm id}
\end{equation*}
is controllable on $SU(n)$ (i.e.\ the special unitary group). Clearly, this implies\footnote{If one
does require controllability of the unitary orbit for all initial states $\rho_0$ then it is actually equivalent
\cite{KDH12}.} that one can control \eqref{eq:control-diss_evolution} on the unitary orbit of $\rho_0$ if
dissipation is switched off. A necessary and sufficient condition for unitary controllability is the well-known
Lie-algebra rank condition \cite{SJ72,JS72,Bro72,Bro73} reading
\begin{equation*}
\expt{{\rm i}H_0, {\rm i}H_j\,|\, j=1,\dots,m}_{\sf Lie} = \su(n)
\quad(\text{or}\;=\mathfrak{u}(n))\,.
\end{equation*}

%

To properly connect this quantum control model with the initial toy model, we will need yet another assumption.

\medskip
\noindent
\textbf{Assumption IN:} The set of diagonal density matrices
$$
\posdiag{n}:=\{ \operatorname{diag}(x)\in\mathbb R^{n\times n}\,|\, x\in\Delta^{n-1}\}
$$
is invariant under the semiflow $(e^{-t\hGA})_{t\in\mathbb R_+}$ and thus $\Delta^{n-1}$ is also invariant under
the semiflow $(e^{-tB_0(\hGA)})_{t\in\mathbb R_+}$. 

\medskip
\noindent
Since $(e^{-t\hGA})_{t\in\mathbb R_+}$ is positive and trace preserving by construction, the invariance of 
$\posdiag{n}$ under $(e^{-t\hGA})_{t\in\mathbb R_+}$ boils down to the obvious condition that $\hGA$ maps 
diagonal matrices to diagonal matrices.

\medskip

\noindent
{\bf Conclusion for Quantum Systems:}
Finally, by means of UC+S (or FUC) and IN, it is easy to verify that the closure of the unitary orbit of 
$\mathfrak{reach}_\Lambda(x_0)$ (more precisely, the image of $\mathfrak{reach}_\Lambda(x_0)$ under the 
operator $x \mapsto \diag(x)$) is contained in the closure of the reachable set
$\mathfrak{reach}_\Sigma(U\diag(x_0)U^\dagger)$. We elaborate this idea further in Cor.~\ref{cor:toy_quantum} below.

\medskip
Other authors used quite similar ideas to investigate reachable sets of quantum-dynamical control systems
\cite{Khaneja01b,Yuan10,rooney2018}. In particular, in \cite{rooney2018} the authors restrict themselves 
to a subsimplex of the standard simplex (which results from a Weyl-chamber type of construction) in order 
to eliminate ambiguities which result from different orderings of the eigenvalues of a density matrix. 
Moreover, their setting is more general as they avoid the invariance condition IN. However, the resulting 
conditions are hard to verify for higher-dimensional systems.

\bigskip

\section{Toy Models with\\Unique Attractive Fixed Point}
\label{sec:particular_models}




\noindent Models with a unique attractive fixed point are of 
particular interest for applications. Thus we introduce the terminology \emph{relaxing} for $\hGA$, if there 
exists a $\rho_\infty\in \pos{n}$ such that
\begin{equation}\label{eq:relaxing_semigroup}
\lim_{t\to\infty} e^{-t\hGA}\rho = \rho_\infty
\end{equation}
for all $ \rho \in \pos{n}$. 

Our first results show that there exists a rich class of physically relevant models 
motivated by quantum dynamical qubit systems \cite[Eq.~(B30)]{BSH16} which are relaxing and satisfy 
the invariance condition IN.

\begin{lemma}\label{lemma_semigroup_relaxing}
Let $n\in\mathbb N$ be arbitrary and consider
\begin{equation*}
N_+:=\sum_{j=1}^{n-1}a_j e_{j}e_{j+1}^T \quad\text{and}\quad N_-:=\sum_{j=1}^{n-1}b_j e_{j+1}e_{j}^T
\end{equation*}
with arbitrary $a_1,\ldots,a_{n-1},b_1,\ldots,b_{n-1}\in\mathbb R$ and $(e_j)_{j=1}^n$ being the standard basis
of $\mathbb C^n$. Then the operator $\hGA_N$ induced by $V_1:=N_+$ and $V_2:=N_-$ via 
\eqref{eq:lindblad_V} satisfies the following:
\begin{itemize}
\item[(i)] $\hGA_N$ fulfills {\rm IN}. Moreover, its matrix representation 
on diagonal matrices (with respect to the standard identification $x \to \diag(x)$) is given by
\end{itemize}
\begin{equation}\label{eq:action_gamma_Delta}
\begin{split}
B_0 & = \sum_{j=1}^{n-1} a_j^2 (e_{j+1}-e_j)e_{j+1}^T + b_j^2 (e_j-e_{j+1})e_j^T\\
&= \begin{pmatrix} b_1^2&-a_1^2&&\\-b_1^2&a_1^2+b_2^2&-a_2^2&\\&-b_2^2&a_2^2+b_3^2&-a_3^2\\&&-b_3^2&\ddots \end{pmatrix}\in\mathbb R^{n\times n}\,.
\end{split}
\end{equation}
\begin{itemize}
\item[(ii)] If $a_1,\ldots,a_{n-1},b_1,\ldots,b_{n-1} \neq 0$ then $B_0$ is relaxing on $\Delta^{n-1}$, i.e.~there exists 
a unique $x_\infty\in\Delta^{n-1}$, $x_\infty>0$ such that $\lim_{t\to\infty}e^{-tB_0}x=x_\infty$ for all $x\in\Delta^{n-1}$.
\end{itemize}
\end{lemma}

\begin{proof}
Let $j,k\in\lbrace 1,\ldots,n\rbrace$ and $Y\in\mathbb C^{n\times n}$. A straightforward computation yields
\begin{align*}
\big(\hGA_N&(Y)\big)_{jk} = e_j^T\hGA_N(Y)e_k\\
= & \frac12(a_{j-1}^2+a_{k-1}^2+b_j^2+b_k^2)Y_{jk}\\
& -a_ja_kY_{(j+1)(k+1)}-b_{j-1}b_{k-1}Y_{(j-1)(k-1)}\,.
\end{align*}
This readily implies (i).
Statement (ii) can be shown via the Perron-Frobenius theorem as follows. Let $t>0$ be arbitrary. By 
\eqref{eq:action_gamma_Delta} there exists $c\in\mathbb R_+$ such that all entries
of $ctI_n-tB_0$ are non-negative (denoted by $ctI_n-tB_0 \geq 0)$. This is still true if we take any 
power of $ctI_n-tB_0$ and due to
$a_j,b_j \neq 0$, evidently, $(ctI_n-tB_0)^{n-1}>0$ (positive entries) so
$$
0<e^{ctI_n-tB_0}=e^{ctI_n}e^{-tB_0}=e^{ct}e^{-tB_0}
$$
and thus $e^{-tB_0}>0$. Furthermore, $e^{-tB_0}$ has spectral radius 
one---this follows from \cite[Thm.~8.1.22]{HJ1} due to ${\bf 1}^TB_0=0 $ which 
implies ${\bf 1}^T e^{-tB_0}={\bf 1}^T $, i.e.~$e^{-tB_0}$ leaves $\Delta^{n-1}$ invariant.
Moreover, one can show \cite[Thm.~8.2.11]{HJ1} that $0$ is a simple eigenvalue and every
other eigenvalue of $-tB_0$ has strictly negative real part. Using the Jordan canonical form of $-tB_0$ 
this readily implies convergence of $e^{-t B_0}$ to a matrix of rank one as $t\to\infty$. 
By an argument similar to the one given in Lemma \ref{lemma_row_conv} there exists $x_\infty\in\Delta^{n-1}$, $x_\infty>0$ such that $ e^{-tB_0}\to x_\infty{\bf 1}^T$ as $t\to\infty$, cf.~\cite[Thm.~8.2.11]{HJ1}.
\end{proof}
\noindent Equivalent to (ii) of the previous lemma is the statement that $\hGA_N$ is relaxing on $\posdiag{n}$. 
 In fact, one can show (cf.~\cite{Fagnola2015}) that $\hGA_N$ is actually 
relaxing on all of $\pos n$.

\medskip
Here and henceforth, 
let
\begin{equation}\label{eq:sigma}
\sigma_+ := \sum_{k=1}^{n-1}\sqrt{k(n-k)} e_ke_{k+1}^T
\quad\text{and}\quad
\sigma_- := (\sigma_+)^T
\end{equation}
be the ladder operators in spin-$j$ representation giving rise to $n=2j+1$ levels for half-integer (fermionic)
and integer (bosonic) spin quantum numbers $j\in\{\tfrac{1}{2},1,\tfrac{3}{2},2,\dots\;\}$.

\begin{proposition}\label{thm_bath}
Let $n\in\mathbb N$ and $d\in\Delta^{n-1}$, $d>0$. Moreover,
\begin{equation*}
\sigma_+^d:=\sum\nolimits_{k=1}^{n-1}\sqrt{k(n-k)}\cos(\theta_k) e_{k}e_{k+1}^T
\end{equation*}
and
\begin{equation*}
\sigma_-^d:=\sum\nolimits_{k=1}^{n-1}\sqrt{k(n-k)}\sin(\theta_k) e_{k+1}e_{k}^T\,,
\end{equation*}
where
\begin{equation}\label{eq:thermal_angle}
\theta_k:=\arccos\Big(\big({1+\frac{d_{k+1}}{d_k}}\big)^{-\frac{1}{2}} \Big)\in\Big(0,\frac{\pi}{2}\Big)
\end{equation}
for $k= 1,\ldots,n-1$. Then $\hGA_{d}$ induced by $V_1:=\sigma_+^d$ and $V_2:=\sigma_-^d$ via 
\eqref{eq:lindblad_V} satisfies {\rm IN} and the generated semigroup $(e^{-tB_0(\hGA_d)})_{t\in\mathbb R_+}$ is relaxing on $\Delta^{n-1}$ into $d$.
\end{proposition}

\begin{proof}
Obviously, we can apply Lemma \ref{lemma_semigroup_relaxing} with
$$
a_k:=\sqrt{\frac{k(n-k)d_k}{d_k+d_{k+1}}} \quad\text{and}\quad b_k:=\sqrt{\frac{k(n-k)d_{k+1}}{d_k+d_{k+1}}}
$$
for all $k= 1,\ldots,n-1$. Now the only thing left to prove is
$d\in\operatorname{ker}(B_0)$ which implies that $d$ is the 
unique attractive fixed point of $(e^{-tB_0})_{t \in \mathbb R_+}$. By means of \eqref{eq:action_gamma_Delta}
one immediately gets $B_0d=0$. 
\end{proof}

\subsection*{Fixed Points for Given (Bath) Temperatures }

Recall how the temperature $T > 0$ given as a macroscopic parameter of a bath  relates to the 
equilibrium state $\rho_\text{Gibbs}$ (henceforth called \emph{Gibbs state}) of an \mbox{$n$-level} quantum
system with Hamiltonian $H_0$ once the system is `\/opened\/' by coupling it to the bath and letting it
equilibrate.

In equilibrium, the quantum system is assumed to adopt the bath temperature in the sense that the Gibbs
state $\rho_\text{Gibbs}$ exhibits the same eigenbasis as $H_0$ and its corresponding eigenvalues can be
interpreted as populations of the energy levels of $H_0$ following the Boltzmann distribution: 
$$ 
\frac{\lambda(\rho_{\sf Gibbs})_{k}} {\lambda(\rho_{\sf Gibbs})_{k'}} = \frac{e^{-E_k/T} } {e^{-E_{k'}/T} }
$$
for $k,k'=1,2,\dots, n$ and $T>0$. This obviously
leads to 
$$
\rho_\text{Gibbs}=\frac{e^{-H_0/T}}{\operatorname{tr}(e^{-H_0/T})}
$$
(see, e.g., \cite{AlickiLendi07}). If $H_0$ is diagonal 
(which we will assume in the following w.l.o.g.) this boils down to $\rho_{\sf Gibbs}=\diag(d)$ with
\textit{Gibbs vector}
\begin{equation}\label{eq:gibbs_vec}
d := \frac{(e^{-E_k/T})_{k=1}^n}{\sum_{k=1}^n e^{-E_k/T}}\in\Delta^{n-1}\,.
\end{equation}
Note that different $T > 0$ and $H_0$ may lead to the same Gibbs state/vector.

The equilibration itself can be described as a Markovian relaxation process following the
\gks-equation~\eqref{eq:lindblad_V} with $V_1:=\sigma_+^d$ and $V_2:=\sigma_-^d$ given in Prop.~\ref{thm_bath}.
To this end, $\sigma_+^d$ and $\sigma_-^d$ are designed to guarantee that $\rho_{\sf Gibbs}=\diag(d)$ is
the unique fixed point\footnote{Of course, if $H_0$ and thus $\rho_{\sf Gibbs}$ are not diagonal one has to
adjust the construction of $\sigma_+^d$ and $\sigma_-^d$ by replacing $e_k$ by the corresponding
eigenvector to $H_0$.\label{footnote_H_0_gen_diag}} of the equilibration. 
Roughly speaking, $\sigma_+^d,\sigma_-^d$ can be interpreted to model the transition rates between
neighbouring energy levels. For this to work without ``physically'' forbidden jumps we have to require
that the energy levels of $H_0$ and thus the resulting Gibbs vector $d=d(T)\in\Delta^{n-1}$ are ordered: 
w.l.o.g.~we assume $E_1\leq\ldots\leq E_n$ to be increasing and therefore $d_1\geq d_2\geq\ldots\geq d_n$
to be decreasing.

In the sequel, we want to analyse how $(e^{-t\hGA_d})_{t \in \mathbb R_+}$ (cf.~Prop.~\ref{thm_bath}) behaves for
different choices of the Gibbs vector, i.e.~for different $H_0$ and at different temperatures $T$. 
The following scenarios are of special interest:

\medskip
\noindent
\textbf{Equidistant energy levels:} If the neighbouring ratios $\frac{d_{k+1}}{d_{k}}$ 
are constant for all $k$ (which obviously corresponds to equidistant energy levels $E_k$) so $\theta_k=\text{const}=:\theta$ in \eqref{eq:thermal_angle}, then the generators $\sigma_+^d,\sigma_-^d$ become $\cos(\theta)\sigma_+$, 
$\sin(\theta)\sigma_-$.

\medskip
\noindent
\textbf{High-temperature limit:} The case $d = \mathbf{1}/n$ (obtained via taking the limit $T\to\infty$ in \eqref{eq:gibbs_vec}) yields $\cos(\theta_k)=\sin(\theta_k)=\frac{1}{\sqrt{2}}$ for all $k= 1,\ldots,n-1$ so the generators $\sigma_+^d,\sigma_-^d$ 
become $\sigma_+,\sigma_-$ (up to a global factor). 

\medskip
\noindent 
\textbf{Low-temperature limit:} If the entries of $d$ are sorted and distinct, i.e.~$d_1>d_2>\ldots>d_n$, then $d$ becomes $e_1$ when taking the limit $T\to 0^+$ in \eqref{eq:gibbs_vec}---hence $\sigma_-^d\to 0$ and $\sigma_+^d\to \sigma_+$ so it is enough to consider only one generator.

\bigskip
\section{Proof of Main Results}
\label{sec:main_results}

\subsection{Reachability Results in the Low-Temperature Limit}\label{subsec:reach_zero}
\subsubsection{Global Noise}\label{subsec:global}
In this section, we consider noise on a single qudit in the low-temperature limit, i.e.~a single $n$-level system 
with generator $ \sigma_+$ in \eqref{eq:lindblad_V}. \vspace{6pt}

\noindent\textbf{Theorem \ref{thm_1}.} {\it Let $n\in\mathbb N$ be arbitrary and consider $\hGA_0$ induced by a
single generator $\sigma_+$ via \eqref{eq:lindblad_V}. Then for the toy model $\Lambda_0$ from Sec.~\ref{sec:toy_model}
with $B_0(\hGA_0)$, the closure of the reachable set of any initial state $x_0 \in \Delta^{n-1}$ exhausts the whole
standard simplex, i.e.}
\begin{equation*}
\overline{\mathfrak{reach}_{\Lambda_0}(x_0)}=\Delta^{n-1}\,.
\end{equation*}

\noindent
To prove this, we first need the following auxiliary results.

\begin{lemma}\label{lemma_row_conv}
Let $n\in\mathbb N$ and $c_1,\ldots,c_{n-1}>0$. Then for
\begin{equation}\label{eq:matrix_A}
A:=\begin{pmatrix}
0&-c_1&0&\ldots&0 \\
0&c_1&-c_2&\ddots&\vdots \\
\vdots&\ddots&c_2&\ddots&0 \\
\vdots & &\ddots &\ddots&-c_{n-1} \\
0&\ldots&\ldots&0&c_{n-1}
\end{pmatrix}\in\mathbb R^{n\times n}\,,
\end{equation}
one has $\lim_{t\to\infty}\exp(-tA) = e_1\mathbf{1}^T$, so the resulting matrix has ones in the first
row and all other entries are zero.
\end{lemma}

\begin{proof}
Obviously the above statement is related to, but not a special case of Lemma \ref{lemma_semigroup_relaxing}.
Consider the following block-decomposition
\begin{equation*}
A = \begin{pmatrix} 0& A_{12}\\ 0&A_{22}\end{pmatrix} \quad\text{with}\quad 
A_{22}\in\mathbb R^{(n-1) \times (n-1)}
\end{equation*}
and note that $t \mapsto \Phi(t) := \exp(-tA)$ satisfies the {\sc ode} $\dot{\Phi}(t) = -A \Phi(t)$
with $\Phi(0) = I_n$. Now decomposing $\Phi(t)$ in the same way as $A$ and taking into account
that $\Phi(t)$ satisfies the above {\sc ode} readily yields the following representation
\begin{equation*}
\Phi(t) = \begin{pmatrix} \Phi_{11}(t) & \Phi_{12}(t)\\0 & \Phi_{22}(t) \end{pmatrix} 
\end{equation*}
with $\Phi_{22}(t) = \exp(-tA_{22})$ and $\Phi_{11}(t) = 1$. Finally, via the variation
of parameters formula we obtain
\begin{equation*}
\begin{split}
\Phi_{12}(t) & = -\int_0^t A_{12} \exp(-(t-s)A_{22})\,{\rm d}s \\
& = -A_{12} \big[A^{-1}_{22}\exp(-(t-s)A_{22})\big]_{s=0}^{s=t} \\[2mm]
& = - A_{12}A^{-1}_{22} + A_{12}A^{-1}_{22}\exp(-tA_{22})\,.
\end{split}
\end{equation*}
As $-A_{22}$ is obviously a Hurwitz matrix we conclude 
\begin{equation*}
\lim_{t\to\infty}\exp(-tA) = \lim_{t\to\infty} \Phi(t) = 
\begin{pmatrix} 1 & -A_{12} A^{-1}_{22} \\0 & 0\end{pmatrix} 
\end{equation*}
and the identity $\mathbf{1}^T A = 0$ implies the desired result.
\end{proof}

\begin{lemma}\label{lemma_for_thm_temp_0}
Let $n\in\mathbb N$ be arbitrary and let $A\in\mathbb R^{n\times n}$ be given by \eqref{eq:matrix_A} for some 
$c_1,\ldots,c_{n-1}>0$. 
Then for any $x\in\Delta^{n-1}$ there exist $t_1,\ldots,t_{n-1}\in\mathbb R_+$ and permutation matrices
$\pi_1,\ldots,\pi_{n-1}\in\mathbb R^{n\times n}$ such that
$$
\big(e^{-t_{n-1}A}\pi_{n-1}\ldots e^{-t_1A}\pi_1\big)e_1 = x\,.
$$
\end{lemma}
\begin{proof}[Sketch of Proof]
Note that $\mathbf{1}^TA=0$ guarantees that the hyperplane $\mathbf{1}^T x =1$ is invariant under the flow
$(e^{-tA})_{t\in\mathbb R}$. Moreover, due to the upper triangular structure of $A$, lower-dimensional faces of
$\Delta^{n-1}$ of the form
\begin{align*}
\Delta^{n-1}_{m-1} := \Delta^{m-1}\times \{0_{n-m}\} =\lbrace (y,0_{n-m}) \,|\,y\in\Delta^{m-1}\rbrace 
\end{align*}
are left invariant, too. Now, for $x \neq e_1$ one can consider the backward evolution of $x\in\Delta^{n-1}$
and check that, eventually, the trajectory hits a face of $\Delta^{n-1}$ which can be rotated into 
$\Delta^{n-1}_{n-2}\simeq\Delta^{n-2}$ via some permutation $\pi_{n-1}$. Applying this procedure inductively 
$n-1$ times concludes the proof.
\end{proof}

\begin{proof}[Sketch of the Proof of Thm.~\ref{thm_1}]
By Lemma \ref{lemma_semigroup_relaxing}
\begin{equation}\label{eq:6b}
B_0(\hGA_0)=\sum\nolimits_{j=1}^{n-1} j(n-j) (e_{j+1}-e_j)e_{j+1}^T
\end{equation}
so we may apply Lemma \ref{lemma_row_conv} and \ref{lemma_for_thm_temp_0} to $B_0$. The former one,
in particular, implies that for arbitrary $y\in\mathbb R^n$
one has
\begin{equation}\label{eq:noise_inf_time}
e^{-t B_0}(y)\overset{t\to\infty}\longrightarrow(\mathbf{1}^T y )e_1
\end{equation}
and thus we find $t\geq 0$ such that
\begin{equation}\label{eq:relax_approx_1}
\|({\bf 1}^T y)e_1-e^{-sB_0}(y)\|_1<\varepsilon\quad\text{ for all }s\geq t\,.
\end{equation}

Now, let $\varepsilon > 0$ and $x_0,x\in\Delta^{n-1}$. The remaining proof consists of the following
two steps shown here:
\begin{equation}\label{eq:steps_idea}
x_0 \overset{\text{Step }1}\longrightarrow e_1  \overset{\text{Step }2}\longrightarrow x\,.
\end{equation}
We have to find $ x_F\in \mathfrak{reach}_{\Lambda_0}( x_0)$ s.t.~$\| x- x_F\|_1<\varepsilon$.
Step 1 is about relaxation of the diagonal system to the ground state $e_1$ by applying 
$e^{-tB_0}$ in the limit $t\to\infty$, cf.~\eqref{eq:noise_inf_time} and \eqref{eq:relax_approx_1}. 
Step 2 exploits the fact that from the ground state $e_1$, one can reach 
any other diagonal state $ x$ via $(e^{-tB_0})_{t \in \mathbb R_+}$ and suitable permutations $\pi$
in finite time (i.e.~within $\mathfrak{reach}_{\Lambda_0}$), cf.~Lemma \ref{lemma_for_thm_temp_0}. This is sufficient to perform the scheme suggested 
in \eqref{eq:steps_idea} with arbitrary precision so $ x\in\overline{\mathfrak{reach}_{\Lambda_0}(x_0)}$.
\end{proof}

\begin{remark}\label{rem_reach_exact}
Be aware that Step 2 in the proof of Thm.~\ref{thm_1} is ``exact'' in the sense that starting from the ground
state $e_1$ (in the model $\Lambda_0$), one can reach every element of $\Delta^{n-1}$ in finite time, cf.~Lemma \ref{lemma_for_thm_temp_0}.
\end{remark}\medskip

\subsubsection{Local Noise Coupling}
In this section, we consider local noise of temperature zero and a finite number of qudits, i.e.~a ``chain'' 
of $n$-level systems (of length $m$) with generators of the form $ I\otimes \sigma_+$ in \eqref{eq:lindblad_V}. 

\medskip

\noindent\textbf{Theorem \ref{thm_2}.} {\it 
Let $m,n\in\mathbb N$ be arbitrary and let $\hGA_{0,loc}$ 
be solely generated by $ I_{n^{m-1}}\otimes\sigma_+$ via \eqref{eq:lindblad_V}.
Then for toy model $\Lambda_{0,loc}$ from Section \ref{sec:toy_model} with $B_0(\hGA_{0,loc})$,
the closure of the reachable set of any initial state $ x_0\in\Delta^{n^m-1}$ exhausts the whole standard simplex, i.e.}
$$
\overline{\mathfrak{reach}_{\Lambda_{0,loc}}( x_0)}=\Delta^{n^m-1}\,.
$$

For the proof of this theorem, the following auxiliary result is of importance.

\begin{lemma}\label{lemma_reach_oplus}
Let $k\in\mathbb N$, $\alpha_1,\ldots,\alpha_k\in\mathbb N\setminus\{1\}$ and $Y_j\in\mathbb R^{\alpha_j\times\alpha_j}$
for $j=1,\ldots,k$. Consider the toy models $\Lambda_1,\ldots,\Lambda_k$ obtained from Section 
\ref{sec:toy_model} by replacing $B_0$ by $Y_1,\ldots,Y_k$, respectively, and  
assume the following:\vspace{2pt}
\begin{itemize}
\item[(a)] Starting from the ground state of the individual systems, every other state can be reached (in finite time).
More precisely, $\mathfrak{reach}_{\Lambda_j}(e_1)=\Delta^{\alpha_j-1}$ for all $j$.\vspace{3pt}
\item[(b)] $Y_je_1=0$ for all $j=1,\ldots,k$.\vspace{2pt}
\end{itemize}
Then the toy model $\Lambda_{\text{diag}}$ from Sec.~\ref{sec:toy_model} with
$$B_0 :=\operatorname{diag}(Y_1,Y_2,\ldots,Y_k)
\in\mathbb R^{(\alpha_1+\ldots+\alpha_k)\times (\alpha_1+\ldots+\alpha_k)}$$
admits
\begin{equation}\label{eq:reach_oplus}
\mathfrak{reach}_{\Lambda_\text{diag}}(e_1)=\Delta^{\alpha_1+\ldots+\alpha_k-1}\,.
\end{equation}
\end{lemma}

\begin{proof}[Sketch of Proof]
We only have to prove this for $k=2$ as the general case can be obtained by induction. 

First, note that starting from $e_{1}$ one can reach every state 
of the form $(r e_{1},(1-r)e_{1})\in\mathbb R^{\alpha_1}\times\mathbb R^{\alpha_2}=\mathbb R^{\alpha_1+\alpha_2}$ with $r\in [0,1]$. This is easily
achieved via (a) and appropriate permutations.


Secondly, consider an arbitrary target $ x\in\Delta^{\alpha_1+\alpha_2-1}$ which of course can be 
decomposed into $ x= (x_1,x_2)$ with $ x_j\in\mathbb R_+^{\alpha_j}$.
Again by (a) we know that there exits switching sequences and permutations such that
the dissipation operator $Y_j$ interlaced with these permutations drives 
$({\bf 1}^T x_j)e_{1}$ to $ x_j$ in time $t_j\in\mathbb R_+$ for $j=1,2$. 
Assume w.l.o.g.~$t_1\geq t_2$. 

Then starting from
$(({\bf 1}^T x_1)e_{1}, ({\bf 1}^T x_2)e_{1})$ 
the control scheme goes as follows: 
Run on $({\bf 1}^T x_1)e_{1}$ the switching sequence which steers to 
$ x_1$ in time $t_1$. Stay in $({\bf 1}^T x_2)e_{1}$ till $(t_1-t_2)$ which is 
possible by (b) and then, for the remaining time, run in parallel on the second system the (shifted) switching sequence which steers to 
$ x_2$. Thus at time $t_1$ we reach $ (x_1,  x_2)= x$ which concludes the proof. 
\end{proof}

\begin{proof}[Proof of Thm.~\ref{thm_2}]
The case $n=1$ is covered by Thm.\ref{thm_1} so we may assume $n>1$. 

Let $\varepsilon>0$ and $ x_0,x \in\Delta^{n^m-1}$. We have to find $ x_F\in \mathfrak{reach}_{\Lambda_{0,loc}}
( x_0)$ such that $\| x- x_F\|_1<\varepsilon$. The proof, similar to that of Thm.~\ref{thm_1}, 
consists of the following steps:
\begin{equation}\label{eq:steps_idea_2}
 x_0\overset{\text{Step }1}\longrightarrow e_{1} \overset{\text{Step }2}
\longrightarrow  x\,.
\end{equation}


For applying Lemma \ref{lemma_reach_oplus} in Step 2 check that
$Y_j=B_0(\hGA_0)$ from \eqref{eq:6b} for $j =1, \dots, n^{m-1}$
satisfies conditions (a) and (b), which obviously hold due to Thm.~\ref{thm_1}, Remark 
\ref{rem_reach_exact} and Eq.~\eqref{eq:6b}\footnote{Here we use $I_k\otimes\sigma_+ = \operatorname{diag}(\sigma_+,\ldots,\sigma_+)$ which for any $X\in\posdiag{kn}$ (when decomposed into $X=\operatorname{diag}(X_1,\ldots,X_k)$) implies $\hGA_{I_k\otimes\sigma_+}(X)=\operatorname{diag}(\hGA_0(X_1),\ldots,\hGA_0(X_k))$ and thus $B_0(\hGA_{I_k\otimes\sigma_+})= \operatorname{diag}(B_0(\hGA_0),\ldots,B_0(\hGA_0))$ as is readily verified.}.

%
%
Thus we know $\mathfrak{reach}_{\Lambda_{0,loc}}(e_1)=\Delta^{n^m-1}$ and in 
particular $ x\in \mathfrak{reach}_{\Lambda_{0,loc}}(e_1)$. 
For the first step in \eqref{eq:steps_idea_2}, we may decompose $ x_0$ into $( x_{1},\ldots,x_{n^{m-1}})$ with $ x_j\in\mathbb R_+^{n}$. Then
\begin{align*}
\lim_{t\to\infty}e^{-tB_0}x_0&=\lim_{t\to\infty}\big(e^{-tB_0(\hGA_0)} x_{1},\ldots,e^{-tB_0(\hGA_0)} x_{n^{m-1}}\big)\\
&=\big( ({\bf 1}^T x_{1})e_1, \ldots, ({\bf 1}^T x_{n^{m-1}})e_1\big)
\end{align*}
by \eqref{eq:noise_inf_time}. Then applying an appropriate permutation yields
$$
\big(  ({\bf 1}^T x_{1}),\ldots,({\bf 1}^T x_{n^{m-1}}),0_{n^m-n^{m-1}}\big)\in\Delta^{n^m-1}\,.
$$
Repeating this scheme $m$ times in total leaves us with
$$
\Big(  \sum\nolimits_{j=1}^{n^{m-1}}({\bf 1}^T x_{j}),0_{n^m-1}\Big)=e_1
$$
because $ x_0\in\Delta^{n^m-1}$ so ${\bf 1}^T x_0=\sum_{j=1}^{n^{m-1}}{\bf 1}^T  x_{j}=1$. 
Clearly, above limits (for $t\to\infty$) cannot be reached exactly, 
yet again by \eqref{eq:noise_inf_time}, for every $y\in\mathbb R^{n}$ we find 
$t \geq 0$ such that
$$
\|({\bf 1}^T y)e_1-e^{-sB_0(\hGA_0)}(y)\|_1<\tfrac{\varepsilon}{m\cdot n^{m-1}}\,.
$$
for all $s \geq t$. This yields $ x_F\in \mathfrak{reach}_{\Lambda_{0,loc}}( x_0)$ with 
$$
\|e_1- x_0^F\|_1<m\cdot\big( n^{m-1}\cdot\tfrac{\varepsilon}{m\cdot n^{m-1}}\big)=\varepsilon
$$ 
as each of the $m$ relaxation steps has precision $\frac{\varepsilon}{m}$.
\end{proof}

Let us quickly describe how the previous theorems pertain to the quantum realm.

\begin{corollary}\label{cor:toy_quantum}
Let $\Sigma$  be a coherently controlled quantum system of the form (5), which satisfies
Assumption UC+S or FUC, as well as IN. Moreover, let $\hGA_0$ or $\hGA_{0,loc}$ be given as above.
Then one can (approximately)
reach every other state, i.e. 
$$\overline{\mathfrak{reach}_{\Sigma}(\rho_0)}=\pos n\,.$$
\end{corollary}
\begin{proof}[Sketch of Proof]
Starting from any 
$\rho_0\in\pos n$, due to Assumption UC+S (where the {\gks} noise is switched off) or FUC we can unitarily 
transform $\rho_0\to U\rho_0U^\dagger$ such that it is diagonal in an eigenbasis of $H_0$. This in particular 
means $[H_0,U\rho_0U^\dagger]=0$ so we are in the diagonal case (cf.~footnote \ref{footnote_H_0_gen_diag}) 
with effectively no coherent drift but only dissipation and coherent controls, 
i.e.~in the realm of the toy model via the obvious one-to-one correspondence $\posdiag n \leftrightarrow
\Delta^{n-1}$. Here one can (approximately) reach every other diagonal state which 
by finally rotating back gives the desired result for the respective 
control system $\Sigma$.
\end{proof}


\subsection{Reachability Results for Non-Zero Temperature}

Now for all temperatures $T\in[0,\infty]$, in the qubit case (bath coupling $\Sigma_d$ with UC+S or FUC) the
closure of the reachable set for any initial state $ \rho_0\in\pos2$ equals
$$
\lbrace \rho\in\pos{2}\,|\, \rho\prec  \rho_\text{Gibbs} \vee  \rho\prec \rho_0\rbrace
$$
as can be seen easily, cf.~\cite{RBR12}. One might hope that this extends to general $n$-level systems with $n>2$
at finite temperatures. However, this is not true even if the above is taken as an upper bound for the reachable
set, as the following example shows.

\begin{example}
Let
$$
d=\frac1{e^{0.64}+1+e^{-0.64}}\begin{pmatrix} e^{0.64}\\1\\e^{-0.64} \end{pmatrix}\approx\begin{pmatrix}0.5539\\0.2921\\0.1540
\end{pmatrix}\in\Delta^2\,.
$$ 
Then for $ \rho_0 = \operatorname{diag}(0.55,0.4,0.05) \in\pos 3$ and the semigroup 
$(e^{-t\hGA_d})_{t\in\mathbb R_+}$ (cf.~Prop.~\ref{thm_bath}) one gets for $t=1$
$$
e^{-\hGA_d}( \rho_0)=\operatorname{diag}\Big(e^{-B_0(\hGA_d)}\begin{pmatrix}0.55\\0.4\\0.05  \end{pmatrix}\Big)\approx\operatorname{diag}\begin{pmatrix}0.5783\\0.3098\\0.1119\end{pmatrix}\,.
$$
Evidently, $e^{-\hGA_d}( \rho_0)\not\prec \operatorname{diag}(d) = \rho_\text{Gibbs}$ and 
$e^{-\hGA_d}( \rho_0)\not\prec  \rho_0$.
\end{example}

To obtain some analytic results we restrict ourselves to the case of equidistant energy
levels (cf.~Sec.~\ref{sec:particular_models}). Thus $d$ is of the form
\begin{equation}\label{eq:equidist_d_vec}
d=\frac{1-\alpha}{1-\alpha^n}\begin{pmatrix} 1\\\alpha\\\vdots\\\alpha^{n-1} \end{pmatrix}
\end{equation}
for some $\alpha\in(0,1)$. This includes the so-called diagonal spin case. \vspace{6pt}

\noindent\textbf{Theorem \ref{thm_3}.} {\it Let $n\in\mathbb N$ and $d\in\Delta^{n-1}$ such that $\frac{d_{j+1}}{d_j}$
is constant for $j=1,\ldots,n-1$. Also let $\hGA_d$ be induced by $\sigma_-^d,\sigma_+^d$ (cf.~Prop.~\ref{thm_bath}). 
Then the reachable set for the toy model $\Lambda_d$ with $B_0(\hGA_d)$
satisfies}
\begin{equation*}
\mathfrak{reach}_{\Lambda_d}(d)\subseteq \lbrace x\in\Delta^{n-1}\,|\, x\prec d \rbrace\,.
\end{equation*}

\noindent
Note that $d$ is the unique fixed point of $(e^{-tB_0(\hGA_d)})_{t \in \mathbb R_+}$.

\begin{lemma}\label{lemma_perm_partition}
Let $n,k\in\mathbb N$ with $k\leq n$ and let $\pi$ be any permutation on $\lbrace1,\ldots,n\rbrace$. 
Then there exist unique non-empty subsets $\square_1,\ldots,\square_q\subseteq\pi(\lbrace1,\ldots,k\rbrace)$
(henceforth called ``blocks'') 
with the following properties.\vspace{3pt}
\begin{itemize}
\item[(i)] The blocks $\square_1,\ldots,\square_q$ yield a disjoint partition of 
$\pi(\lbrace1,\ldots,k\rbrace)$, i.e.~$\square_{j}\cap\square_l=\emptyset$ for $j \neq l$
and $\bigcup_{j=1}^q \square_j=\pi(\lbrace1,\ldots,k\rbrace)$. \vspace{4pt}
\item[(ii)] The blocks are the ``connected components'' of $\pi(\lbrace1,\ldots,k\rbrace)$. 
More precisely, for each $j\in\lbrace 1,\ldots,q\rbrace$ there exist $b_j^-,b_j^+\in\lbrace 1,\ldots,n \rbrace$ 
such that
$$
\square_j= \lbrace b_j^-, b_j^- +1, \ldots, b_j^+-1, b_j^+\rbrace
$$
and $b_j^--1, b_j^++1 \notin\pi(\lbrace1,\ldots,k\rbrace)$ so the nearest neighbours of the blocks 
are not in $\pi(\lbrace1,\ldots,k\rbrace)$.
\end{itemize}
\end{lemma}

Instead of proving the above lemma, let us quickly illustrate what is going on here by considering an example. 
Then a proof will be evident.
\begin{example}
Let $\pi$ be the permutation (in cycle notation) 
$
\pi = (1,6,2,3,4)(5)
$ 
on $\lbrace 1,\ldots,6\rbrace$. First, consider $k=3$ so $\pi(\lbrace 1,2,3\rbrace)=\lbrace 3,4,6\rbrace$. 
The connected block-components of this set are $\square_1=\lbrace 3,4\rbrace$, $\square_2=\lbrace6\rbrace$ 
which satisfy 
$$
\square_1\cap\square_2 =\emptyset \quad\text{and}\quad \square_1\cup\square_2=\pi(\lbrace 1,2,3\rbrace)
$$ 
and neither of their neighbouring numbers (i.e.~$2,5,7$) are contained within 
$\pi(\lbrace 1,2,3\rbrace)$. To finish off this example, for $k=5$ one gets $\pi(\lbrace 1,2,3,4,5\rbrace)
=\lbrace 1,3,4,5,6\rbrace$. Here, the blocks obviously are 
$\square_1=\lbrace 1\rbrace$, $\square_2=\lbrace3,4,5,6\rbrace$.
\end{example}

\begin{proof}[Proof of Thm.~\ref{thm_3}]
Using \eqref{eq:action_gamma_Delta} and \eqref{eq:equidist_d_vec} for $B_0=B_0(\hGA_d)$ gives
\begin{equation}\label{eq:75}
B_0 = \begin{pmatrix} c_1\alpha&-c_1&&&&\\
-c_1\alpha&c_1+c_2\alpha&-c_2&&\\
&-c_2\alpha&c_2+c_3\alpha&-c_3&\\
&&\ddots&\ddots&\ddots\\
 \end{pmatrix}
\end{equation}
with $c_j:=j(n-j)/(1+\alpha)\geq 0$. In order to show that $\mathfrak{reach}_{\Lambda_d}(d)$ is upper bounded 
by $\lbrace x\in\Delta^{n-1}\,|\, x\prec  d\rbrace$ one has to show that the latter
\begin{itemize}
\item[(a)] contains the initial state.\vspace{4pt}
\item[(b)] is invariant under permutation channels.\vspace{4pt}
\item[(c)] is invariant under the semigroup $(e^{-tB_0})_{t \in \mathbb R_+}$.
\end{itemize}
As (a) and (b) are evident we only have to show (c). As $\exp(-tB_0)$ is linear and the 
set $\lbrace x\in\Delta^{n-1}\,|\, x\prec  d\rbrace$ is a convex 
set, it suffices to prove that the semigroup acts contractively on its extreme points 
$\pi d$, where $\pi$ denotes any permutation matrix.

Thus, we have to show that for every permutation matrix $\pi$ there exists $t_0 > 0$ such that
\begin{equation}\label{eq:76}
\exp(-tB_0)\pi d\prec d \quad \text{for all $t\in [0,t_0)$}\,.
\end{equation}
Again the fact that $\lbrace x\in\Delta^{n-1}\,|\, x\prec  d\rbrace$
is a compact, convex polytope implies that \eqref{eq:76} can be replaced by the tangential condition
\begin{equation}\label{eq:77}
\forall_{\pi\in\mathcal S_n}\, \exists_{\mu>0}\quad( I_n-\mu B_0) \pi d\prec d\,.
\end{equation}
Let $k\in\lbrace 1,\ldots,n\rbrace$ be arbitrary and consider the ``connected components''
$\square_1,\ldots,\square_q$ of the set 
$\pi(\lbrace 1,\ldots,k\rbrace)$, cf.~Lemma 
\ref{lemma_perm_partition}. Then via \eqref{eq:equidist_d_vec} and due to $\alpha\in(0,1)$ the $k$-th 
majorization condition (partial sum condition) entails
\begin{equation}\label{eq:78}
\sum\nolimits_{j=0}^{k-1} \alpha^j=\sum\nolimits_{j=1}^{k} d_j=\sum\nolimits_{j=1}^{k} (\pi d)_{\pi(j)}\,,
\end{equation}
where $\pi$ can be any permutation matrix. The respective partial sum of 
$B_0\pi d$ is given by
$$
\sum\nolimits_{j=1}^{k} (B_0\pi d)_{\pi(j)} = 
\sum\nolimits_{j=1}^q \sum\nolimits_{a\in\square_j} (B_0\pi d)_{a} \,.
$$
If we can show that every $\square_j$-sum individually yields something non-negative, then applying
$( I_n-\mu B_0)$ (for $\mu>0$ sufficiently small) to $\pi d$ can make the 
sum in \eqref{eq:78} only smaller and thus \eqref{eq:77} holds.

Using \eqref{eq:75} and the properties of the $\square_j$ (cf.~Lemma \ref{lemma_perm_partition}),
\begin{align}
&\sum\nolimits_{a\in\square_j} (B_0\pi d)_{a}=\sum\nolimits_{k=b_j^-}^{b_j^+}(B_0\pi d)_{k}\notag\\
&=\sum\nolimits_{k=b_j^-}^{b_j^+}c_{k-1}\big( (\pi d)_{k}-\alpha (\pi d)_{k-1})\notag\\
&\hphantom{=\sum\nolimits_{k=b}}+c_{k}\big(\alpha(\pi d)_{k}- (\pi d)_{k+1})\big)\notag\\
&=c_{b_j^--1}\big((\pi d)_{b_j^-}-\alpha (\pi d)_{b_j^--1}) \label{eq:79a}\\
&\hphantom{=} + c_{b_j^+}\big(\alpha (\pi d)_{b_j^+}-(\pi d)_{b_j^++1})\big)\,. \label{eq:79b}
\end{align}
We know that $b_j^-,b_j^+\in\pi(\lbrace 1,\ldots,k\rbrace)\not\ni b_j^--1,b_j^++1$ so the degrees 
of the involved $\alpha$-monomials satisfy a strict inequality, 
i.e.~$\operatorname{deg}_\alpha((\pi d)_{b_j^-})<\operatorname{deg}_\alpha((\pi d)_{b_j^- -1})$ (and similarly 
for the other one) using \eqref{eq:equidist_d_vec}. Because of $\alpha\in (0,1)$ and $c_j\geq 0$ for all $j$,
the summands involved in \eqref{eq:79a} and \eqref{eq:79b} are non-negative which concludes the
proof.
\end{proof}
%

One might wonder whether it is necessary to restrict oneself to Hamiltonians with equidistant eigenvalues.
The following example gives a positive answer.
\begin{example}
Let $$d=\frac1{ 1+e^{-1/4}+e^{-17/4}}{\begin{pmatrix} 1\\e^{-1/4}\\e^{-17/4} \end{pmatrix}}\approx{\begin{pmatrix}0.5577\\0.4343\\0.0080
\end{pmatrix}}\in\Delta^2$$ 
so the semigroup $( e^{-tB_0})_{t\in\mathbb R_+}$ (cf.~Prop.~\ref{thm_bath}) acts like
$$
e^{-tB_0}{\begin{pmatrix} 0.0080\\0.5577\\0.4343\end{pmatrix} }\approx{\begin{pmatrix}0.0683\\0.5730\\0.3587\end{pmatrix}}\quad\text{  for }t=1/10\,.
$$
Therefore majorization is violated (the largest eigenvalue grows) and the set $\lbrace x\in
\Delta^{n-1}\,|\, x\prec  d\rbrace$ is not left invariant by $(e^{-tB_0})_{t\in\mathbb R_+}$, 
although $d$ satisfies the ``physical'' ordering condition of Sec.~\ref{sec:particular_models}.
\end{example}

\section{CONCLUSIONS}
Within the framework of bilinear control systems, we have described reachable sets for coherently 
controllable Markovian quantum systems in two scenarios: either (i) with switchable Markovian dissipation 
on top of unitary control  or (ii) with (arbitrarily) fast full unitary control and constant Markovian 
noise as drift. In either scenario, the dissipation can be thought of as coupling to a bath of temperature $T$. 

For $T=0$ we have shown that the reachable set encompasses the set of all states (density operators) no matter what 
the initial state is. The result thus generalises previous findings for $m$~qubits \cite{BSH16} to general
$n$-level systems on one hand or general $m$-qu{\em d}it systems on the other. 

For coupling to baths of finite temperatures $T>0$, we have given an inclusion for the reachable set in a 
certain class of initial states. This generalises results on unital dissipative quantum systems \cite{Yuan10}, 
where the bath can be thought of as being in the high-temperature limit $T\to\infty$. 
\addtolength{\textheight}{-81mm}   

Extending the current 
results on finite temperatures $T>0$ to further classes of initial states is an obvious yet 
challenging idea to follow-up.







\bibliographystyle{IEEEtran} 
\bibliography{control21vJan19}

\end{document}